\documentclass[11p,reqno]{amsart}
\usepackage{amsmath}
\usepackage{amsfonts}
\usepackage{amssymb}
\usepackage{graphicx}
\usepackage{color}
\newtheorem{theorem}{Theorem}[section]
\newtheorem*{theorem*}{Theorem}
\newtheorem{lemma}{Lemma}[section]
\newtheorem{corollary}[theorem]{Corollary}
\newtheorem{proposition}{Proposition}[section]

\def\Ric{\text{Ric}}

\def\Ric{\operatorname{Ric}}

\def\W{\operatorname{W}}

\numberwithin{equation}{section}

\begin{document}
	\title[Schwarz Lemmata]{General Schwarz Lemmata  and their applications}

\author{Lei Ni}
\address{Lei Ni. Department of Mathematics, University of California, San Diego, La Jolla, CA 92093, USA}
\email{lni@math.ucsd.edu}

%\subjclass[2010]{32L05, 32Q10, 32Q15, 53C55}
%\keywords{Compact complex manifolds, K\"ahler metrics, positive holomorphic sectional curvature, positive scalar
%\ curvature, projectivized vector bundles}

%\date{}
\begin{abstract} We prove estimates interpolating the Schwarz Lemmata of Royden-Yau and the ones recently established by the author. These more flexible estimates provide additional information on  (algebraic) geometric aspects of compact K\"ahler manifolds with nonnegative holomorphic sectional curvature, nonnegative $\Ric_\ell$ or positive $S_\ell$.

\bigskip

 {\it Dedicated to Professor Luen-Fai Tam on the occasion of his 70th birthday.}

\end{abstract}

\maketitle

\section{Introduction}

There are many generalizations of the classical Schwarz Lemma on holomorphic maps between unit balls via the work of  Ahlfors, Chen-Cheng-Look, Lu, Mok-Yau, Royden, Yau,  etc (see \cite{Kobayashi-H} and \cite{Roy, Yau-sch} and references therein). The one obtained by Royden \cite{Roy} states:
\begin{theorem}\label{thm-sch-roy}
Let $f: M^m\to N^n$ be a holomorphic map. Assume that the holomorphic sectional curvature of $N$, $H(Y)\le -\kappa |Y|^4, \, \forall Y\in T'N$ and the Ricci curvature of $M$, $\Ric^M(X, \overline{X})\ge -K |X|^2, \, \forall X\in T'M$ with $\kappa, K>0$. Let $d=\dim(f(M))$. Then
\begin{equation}\label{eq:sch-roy1}
\|\partial f\|^2(x) \le \frac{2d}{d+1}\frac{K}{\kappa}.
\end{equation}
\end{theorem}

In \cite{Ni-1807} the author proved a new version which only involves the holomorphic sectional curvature of domain and target manifolds.  Recall that for the tangent map $\partial f: T_x'M \to T'_{f(x)}N$ we define its maximum norm square to be
\begin{equation}\label{eq:1}
\|\partial f\|^2_0(x)\doteqdot \sup_{v\ne 0}\frac{|\partial f(v)|^2}{|v|^2}.
\end{equation}

\begin{theorem}\label{thm:sch1} Let $(M, g)$ be a complete K\"ahler manifold such that the holomorphic sectional curvature $H^M(X)/|X|^4 \ge -K$ for some $K\ge0$. Let $(N^n, h)$ be a  K\"ahler manifold such that $H^N(Y)<-\kappa |Y|^4$ for some $\kappa>0$.  Let $f:M\to N$ be  a holomorphic map. Then
\begin{equation}\label{eq:sch-ni}
\|\partial f\|^2_0(x) \le \frac{K}{\kappa}, \forall x\in M,
\end{equation}
provided that the bisectional curvature of $M$ is bounded from below if $M$ is not compact. In particular, if $K=0$, any holomorphic map $f: M\to N$ must be a constant map.
\end{theorem}
 The assumption on the bisectional curvature lower bound can be replaced with the existence of an exhaustion function $\rho(x)$ which satisfies that
\begin{equation}\label{eq:2}
\limsup_{\rho\to \infty} \left(\frac{|\partial \rho|+[\sqrt{-1}\partial \bar{\partial} \rho]_{+}}{\rho}\right)=0.
\end{equation}
The proof  uses   a viscosity consideration from PDE theory. It is also reminiscent of Pogorelov's Lemma \cite{Pogo}  (cf. Lemma 4.1.1 of \cite{Gu}) for Monge-Amp\`ere equation, since the maximum eigenvalue of $\nabla^2 u$ is the $\|\cdot\|_0$ for the normal map $\nabla u$ for any smooth  $u$.
A consequence of Theorem \ref{thm:sch1} asserts that {\it the equivalence of the negative amplitude of the holomorphic sectional curvature implies the equivalence of the metrics}. Namely if $M^m$ admits two K\"ahler metrics $g_1$ and $g_2$ satisfying that
$$
-L_1|X|_{g_1}^4\le H_{g_1}(X)\le -U_1|X|_{g_1}^4, \quad -L_2|X|_{g_2}^4 \le H_{g_2}(X)\le -U_2|X|_{g_2}^4
$$
then for any $v\in T_x'M$ we have the estimates:
$$
|v|^2_{g_2}\le \frac{L_1}{U_2}|v|^2_{g_1};\quad |v|^2_{g_1}\le  \frac{L_2}{U_1}|v|^2_{g_2}.
$$
This result can be viewed as a stability statement of the classical result asserting that a complete K\"ahler manifold with the negative constant holomorphic sectional curvature must be a quotient of the complex hyperbolic space form.  Motivated by Rauch's work which induces much work towards the $1/4$-pinching theorem and, the above stability of K\"ahler metrics it is natural to ask  {\it whether or not  a K\"ahler manifold $M$ with its homomorphic sectional curvature being close to $-1$ is biholomorphic to a  quotient of the complex hyperbolic space.} Besides the  Liouville type theorem for holomorphic maps into manifolds with negative holomorphic sectional curvature, we shall show in Section 5 further implications of this estimate towards the structure of the fundamental groups of manifolds with nonnegative holomorphic sectional curvature.

 Before we state another recent  result of the author we first   recall some basic notions from Grassmann algebra \cite{Fede, Whit}. Let $\mathbb{C}^m$ be a complex Hermitian space (later we will identify the holomorphic tangent spaces $T_x'M$ and $T'_{f(x)}N$ with $\mathbb{C}^m$ and $\mathbb{C}^n$). Let $\wedge^\ell \mathbb{C}^m$ be the spaces of $\ell$-multi-vectors $\{v_1\wedge\cdots \wedge v_\ell\}$ with $v_i\in \mathbb{C}^m$. For ${\bf{a}}=v_1\wedge\cdots \wedge v_\ell, {\bf{b}}=w_1\wedge \cdots w_\ell$, the inner product can be defined as  $\langle {\bf{a}}, \overline{{\bf{b}}}\rangle =\det(\langle v_i, \bar{w}_j\rangle)$. This endows $\wedge^\ell \mathbb{C}^n$ an Hermitian structure, hence a norm $|\cdot|$. There are also other norms, such as the {\it mass} and the {\it comass}, which shall be denoted as $|\cdot|_0$ as in \cite{Whit}, and  could be useful for some problems. We refer \cite{Whit} Sections 13, 14 for detailed discussions. Assume  that $f: (M^m, g)\to (N^n, h)$ is a holomorphic map between two K\"ahler manifolds. Let $\partial f: T'M\to T'N$ be the tangent map. Let $\Lambda^\ell \partial f:\wedge^\ell T_x'M \to \wedge^\ell T_{f(x)}'N$ be the associated map defined as
$\Lambda^\ell \partial f( v_1\wedge \cdots \wedge v_\ell)=\partial f(v_1)\wedge\cdots\wedge \partial f(v_\ell)$. Define $\|\cdot\|_0$ as  $$\|\Lambda^\ell \partial f\|_0(x) \doteqdot\sup_{{\bf a}=v_1\wedge\cdots\wedge v_\ell\ne 0, {\bf a}\in \wedge^\ell T_x'M} \frac{|\Lambda^\ell \partial f({\bf a})|}{|{\bf a}|}.$$
The notion $\|\cdot\|_0$ is adapted to be consistent with the comass notion in \cite{Whit}.
 By the singular value decomposition, we may choose normal coordinates centered at $x_0$ and $f(x_0)$ such that at $x_0$, $df\left(\frac{\partial\, }{\partial z^{\alpha}}\right)=\lambda_\alpha \delta_{i\alpha} \frac{\partial\, }{\partial w^i}$. If we order $\{\lambda_\alpha\}$ such that $|\lambda_1|\ge |\lambda_2|\ge \cdots \ge |\lambda_m|$, $\|\Lambda^\ell \partial f\|_0(x_0)=|\lambda_1\cdots\lambda_\ell|$. It is also easy to see that $\|\partial f \|^2 \doteqdot g^{\alpha\bar{\beta}}h_{i\bar{j}}\frac{\partial f^i}{\partial z^\alpha} \overline{ \frac{\partial f^j}{\partial z^\beta}}=\sum_{\alpha=1}^m |\lambda_\alpha|^2$. The following was proved in Corollary 3.4 of \cite{Ni-1807}.
 \begin{theorem}\label{thm-schni2} Let $f:M^m\to N^n$ ($m\le n$) be a holomorphic map with $M$ being a complete manifold. Assume that  $\Ric^M $ is bounded from below and  the scalar curvature $S^M(x)\ge -K$.  Assume further that   $\Ric^N_m(x)\le -\kappa<0$. Then we have the  estimate
$$
\|\Lambda^m \partial f\|^2_0(x)\le \left(\frac{K}{m\kappa}\right)^m.
$$
\end{theorem}
 Here recall that in \cite{Ni-1807} $\Ric(x, \Sigma)$ is defined as the Ricci curvature of the curvature tensor restricted to the $k$-dimensional subspace $\Sigma\subset T_x'M$. Precisely for any $v\in \Sigma$,  $\Ric(x, \Sigma)(v,\bar{v})\doteqdot \sum_{i=1}^k R(E_i,\overline{E}_i, v,\bar{v})$ with $\{E_i\}$ being a unitary basis of $\Sigma$.  We say that $\Ric_k(x)<0$ if $\Ric(x, \Sigma)<0$  for every k-dimensional subspace $\Sigma$. Clearly $\Ric_k(x)<0$ implies that $S_k(x)<0$,   and  it coincides with $H$ when $k=1$, with the Ricci curvature $\Ric$ when $k=\dim(N)$. Here $S_k(x, \Sigma)$ is defined to be the scalar curvature of the curvature operator restricted to $\Sigma\subset T'_x N$. One can refer to \cite{Ni-1807, Ni, Ni-Zheng2} for the  definitions and related results on the geometric significance of $\Ric_\ell$ and $S_\ell$.

Note that Theorem \ref{thm-schni2} has at least two limits in studying the holomporphic maps. The first it applies only to the case that $\dim(N)$, the dimension of the target manifold is at least as big as the dimension of the domain. The second limit is that it  can only be applied to detect whether or not the map is full-dimensional, namely $\dim(f(M))=\dim(M)$ or not.
The first goal of this paper is to prove a family of estimates for holomorphic maps between K\"ahler manifolds containing the above three results as  special cases. The result below  removes the above mentioned constraints of Theorem \ref{thm-schni2}.

\begin{theorem}\label{thm:main1} Let $f:M^m\to N^n$  be a holomorphic map with $M$ being a complete manifold. When $M$ is noncompact assume either the bisectional curvature is bounded from below or (\ref{eq:2}) holds for some exhaustion function $\rho$. Let $\ell\le \dim(M)$ be a positive integer.

(i) Assume that the holomorphic sectional curvature of $N$, $H^N(Y)\le -\kappa |Y|^4$ and $M$ has, $\Ric_\ell^M\ge -K $, for some $K\ge 0, \kappa >0$. Then
$$
\sigma_\ell(x)\le \frac{2\ell'}{\ell'+1}\frac{K}{\kappa}
$$
where $\sigma_\ell(x)=\sum_{\alpha=1}^\ell |\lambda_\alpha|^2(x)$, and  $\ell'=\min\{\ell, \dim(f(M))\}$. In particular, if $K=0$, the map $f$ must be a constant.

(ii) Assume that  $S^M_\ell(x)\ge -K$ and that   $\Ric^N_\ell(x)\le -\kappa$ for some $K\ge 0, \kappa >0$. Then
$$
\|\Lambda^\ell \partial f\|^2_0(x)\le \left(\frac{K}{\ell\kappa}\right)^\ell.
$$
In particular, if $K=0$, the map $f$ has rank smaller than $\ell$.
\end{theorem}
Note that part (i) above recovers Theorem \ref{thm-sch-roy} for $\ell=\dim(M)$, and recovers Theorem \ref{thm:sch1} for $\ell=1$. Hence it provides a family of estimates interpolating between Theorem \ref{thm-sch-roy}  and \ref{thm:sch1}. Similarly part (ii) recovers Theorem \ref{thm-schni2} when $\ell=\dim(M)$, and recovers Theorem \ref{thm:sch1} for $\ell=1$, noting that in the case $\ell=\dim(M)$, the assumption on the lower bound of bisectional curvature can be weakened to a lower bound of the Ricci curvature (from the proof this is obvious). Hence part (ii) provides a family of estimates interpolating between Theorem  \ref{thm:sch1} and  \ref{thm-schni2}.  Part (ii) also implies that any K\"ahler manifold with $\Ric_\ell\le -\kappa<0$ must be $\ell$-hyperbolic, a result proved in \cite{Ni-1807}. Moreover it can also be applied to $M$ with $\dim(M)>\ell$ or even $\dim(M)>\dim(N)$  concluding more detailed degeneracy information of the map,  re-enforcing the relationship between the  $\ell$ dimensional ``holomorphic" area of  $N$ and the $\Ric^N_\ell$.

The proof of the result (in Section 4) is built upon extensions of $\partial\bar{\partial}$-Bochner formulae of \cite{Ni-1807}, which are proved in Section 3 after some preliminaries in Section 2.  In Section 5 we show that the estimates can be used to rule out the existence  of certain holomorphic mappings under some curvature conditions (cf. Theorem \ref{thm:51}). In particular Theorem \ref{thm:sch1} (cf. Corollary 5.4 of \cite{Ni-1807}) implies that {\it if  a compact K\"ahler manifold $(M, g)$ has $H\ge 0$, then there is no onto homomorphism from its fundamental group  to the fundamental group of any oriented Riemann surface (complex curve) of genus greater than one.} The more flexible Theorem \ref{thm:main1} extends this statement to include all K\"ahler manifolds with $\Ric_\ell \ge0$ (for some $\ell\in \{1, \cdots, m\}$). Note that a similar statement was proved for Riemannian manifold with positive isotropic curvature in \cite{FW}. In \cite{Tsu, Ni} it was proved that if the holomorphic sectional curvature $H>0$ or more generally $\Ric_\ell>0$  then $\pi_1(M)=\{0\}$. The result here provides some information for the nonnegative case. Note that the examples in \cite{Hitchin} indicate that the class of K\"ahler manifolds with $H>0$ (most of them are not Fano) seems to be much larger than that with $\Ric>0$. There has been  very little known for manifold $M$ with  $H\ge 0$ (or $\Ric_\ell\ge 0$ for $\ell<\dim(M)$) comparing with the situation for compact manifolds with $\Ric\ge 0$. In fact when  $M$ is a compact K\"ahler manifold with nonnegative bisectional curvature, Mok's classification result  \cite{Mok} implies that the fundamental group $\pi_1(M)$ must be a Bieberbach one. In  Corollary 5.1 of \cite{Ni-Tam}  a paper by Tam and the author, this was extended (as a result of F. Zheng)  to the case when $M$ is a non-compact complete K\"ahler manifold, but under the nonnegativity of sectional curvature. For compact Riemannian manifolds with nonnegative Ricci curvature Cheeger-Gromoll \cite{CG} proved that $\pi_1(M)$ must be a finite extension of a Bieberbach group. {\it Could this be proven for a compact K\"ahler manifold with $\Ric_\ell \ge 0$ with $\ell<\dim(M)$}? Note that such a statement can not be possibly true for K\"ahler manifold with $B^\perp\ge 0$ (hence nor with $\Ric^\perp\ge 0$).  In a recent preprint \cite{Mu}, the question has been answered positively for $H\ge 0$, assuming additionally that $M$ is a projective variety. Given that there are many non-algebraic K\"ahler manifolds with $H\ge 0$, our result for general K\"ahler manifolds is not contained in \cite{Mu}.

In \cite{ABCKT}, two invariants were defined for a K\"ahler manifold $M$. One is the so-called Albanese dimension $a(M)\doteqdot \dim_{\mathbb{C}}(Alb(M))$ (we use the complex dimension instead), the dimension of the image of  the Albanese map $Alb: M\to \mathbb{C}^{\dim(H^{1,0}(M))}/H_1(M, \mathbb{Z})$. The other invariant is the genus of $M$, $g(M)$ which is defined as the maximal  $\dim(U)$ with $U$ being an isotropic subspace of $H^1(M, \mathbb{C})$. The above consequence of Theorem \ref{thm:51} can be rephrased as that for $M$ with $H^M(X)\ge 0$, or more generally $\Ric_\ell\ge 0$, we must have  $g(M)\le 1$. The same conclusion is obtained  in Section 6 for K\"ahler manifold $M$ with the Picard number $\rho(M)=1$ and $S_2>0$, or $h^{1,1}(M)=1$.  A corollary  of Theorem \ref{thm:51} concludes that if $S^M_\ell>0$, then $a(M)\le \ell-1$. ( This is also a consequence of the  vanishing theorem proved in \cite{Ni-Zheng2}.) These results endow  the curvature $\Ric_\ell$ and $S_\ell$ some algebraic geometric/topological implications.

In Section 5 we also illustrate that the $C^2$-estimate for the complex Monge-Amp\`ere equation is a special case of our computation in Section 3. In Section 6 we  derive some estimates on the minimal ``energy" needed for a non-constant holomorphic map between certain K\"ahler manifolds extending earlier results in \cite{Ni-1807}.

\section{Preliminaries}
We collect some needed algebraic results. For holomorphic map $f: (M^m, g)\to (N^n, h)$, let $\partial f(\frac{\partial\ }{\partial z^{\alpha}})=\sum_{i=1}^n f^i_{\alpha} \frac{\partial\ }{\partial w^i}$ with respect to local coordinates $(z^1, \cdots, z^m)$ and $(w^1, \cdots, w^n)$. The Hermitian form $A_{\alpha\bar{\beta}}dz^\alpha \wedge dz^{\bar{\beta}}$ with
$A_{\alpha\bar{\beta}}=f^i_{\alpha} \overline{f^j_\beta} h_{i\bar{j}}$ is the pull-back of K\"ahler form $\omega_h$ via $f$. By the singular value decomposition for $x_0\in M$ and $f(x_0)\in N$ we may choose  normal coordinates centered at $x_0$ and $f(x_0)$ such that $\partial f(\frac{\partial\ }{\partial z^\alpha})=\lambda_\alpha \delta^i_{\alpha} \frac{\partial\ }{\partial w^i}$. Then $|\lambda_\alpha|$ are the singular values of $\partial f: (T'_{x_0}M, g) \to (T_{f(x_0)}'N, h)$. It is easy to see that $|\lambda_1|^2 \ge \cdots \ge |\lambda_m|^2$ are the eigenvalues of $A$ (with respect to $g$).

\begin{proposition} \label{prop:21} For any $1\le \ell\le m$ the following holds
$$
\sigma_\ell\doteqdot \sum_{\alpha=1}^\ell |\lambda_\alpha|^2 \ge \sum_{1\le \alpha, \beta\le \ell} g^{\alpha \bar{\beta}}A_{\alpha\bar{\beta}}\doteqdot U_\ell.
$$
\end{proposition}
\begin{proof} Arguing invariantly we choose unitary basis of $T'_{x_0}M$ with respect to $g$. Then the left hand side is the partial sum of the eigenvalues of $A$ in descending order, and the right hand side is the trace of the first $\ell\times \ell$ block of $(A_{\alpha\bar{\beta}})$. Hence the result is well-known (cf. \cite{Horn-Johnson}, Corollary 4.3.34).
\end{proof}

For a linear map $L: \mathbb{C}^m\to \mathbb{C}^n$ between two Hermitian linear spaces, $\Lambda^\ell L: \wedge^\ell \mathbb{C}^m \to \wedge^\ell \mathbb{C}^n$ is define as the linear extension of the action on simple vectors:  $\Lambda^\ell L({\bf{a}})\doteqdot L(v_1)\wedge\cdots\wedge L(v_\ell)$ with ${\bf{a}}=v_1\wedge\cdots \wedge v_\ell$. The metric on $\wedge^\ell \mathbb{C}^m$ is defined as $\langle {\bf{a}}, \overline{{\bf{b}}}\rangle =\det(\langle v_i, \bar{w}_j\rangle)$. If $\{e_\alpha\}$ is a unitary frame of $\mathbb{C}^m$, the $\{e_{\lambda}\}$, with  $\lambda=(\alpha_1, \cdots, \alpha_\ell)$, $\alpha_1\le \cdots \le \alpha_\ell$, being the multi-index,    and  $e_{\lambda}=e_{\alpha_1}\wedge \cdots\wedge e_{\alpha_\ell}$, is a unitary frame for $\wedge^\ell \mathbb{C}^m$. The Binet-Cauchy formula implies that this is consistent with the Hermitian product $\langle {\bf{a}}, \overline{{\bf{b}}}\rangle$ defined in the previous section. The norm $\|\Lambda^\ell L\|_0$ is the operator norm with respect to the Hermitian structures of $\wedge^\ell \mathbb{C}^m $ and $\wedge^\ell \mathbb{C}^m$ defined above, which equals to the Jacobian of a Lipschitz map $f$, when $\ell=m$ or $n$, applying to $L=\partial f$ (cf. Section 3.1 of  \cite{Fede}).

For the local Hermitian matrices $A=(A_{\alpha\bar{\beta}})$ and $G=(g_{\alpha\bar{\beta}})$ we denote $A_\ell$ and $G_{\ell}$ be the upper-left $\ell\times \ell$ blocks of them.

\begin{proposition}\label{prop:22} For any $1\le \ell\le m$ the following holds:
\begin{eqnarray}
\|\Lambda^\ell \partial f\|_0^2=\Pi_{\alpha=1}^\ell |\lambda_\alpha|^2&\ge& \frac{\det(A_\ell)}{\det(G_\ell)}\doteqdot W_\ell; \label{eq:21}
\end{eqnarray}
\end{proposition}
\begin{proof} For the inequality in  (\ref{eq:21}), as in the above proposition we may choose a unitary frame of $T'_{x_0}M$ such that $G=\operatorname{id}$. Then the claimed result is also a well-known statement about the partial products of the descending eigenvalues. The result can be seen by applying   4.1.6 of \cite{MM} to $(A+\epsilon G)^{-1}$ and let $\epsilon \to 0$ (see also Problem 4.3.P15 of \cite{Horn-Johnson}).

For the equality (\ref{eq:21}), first observe that
$$
\|\Lambda^\ell \partial f\|^2_0(x) \ge \frac{| \partial f\left(v_1\right)\wedge \cdots \wedge \partial f\left(v_\ell\right)|^2}{|v_1\wedge\cdots \wedge v_\ell|^2}=\Pi_{\alpha=1}^\ell |\lambda_\alpha|^2
$$
if $\{v_\alpha \}$ are the eigenvectors of $A$ with eigenvalues $\{|\lambda_\alpha|^2\}$. On the other hand for general orthonormal vectors $\{v_\alpha\}$, the above paragraph implies $  \frac{| \partial f\left(v_1\right)\wedge \cdots \wedge \partial f\left(v_\ell\right)|^2}{|v_1\wedge\cdots \wedge v_\ell|^2}\le \Pi_{\alpha=1}^\ell |\lambda_\alpha|^2
$. Combining them we have the equality in  (\ref{eq:21}). \end{proof}

\section{$\partial\bar{\partial}$-Bochner formulae}
Here we generalize the $\partial\bar{\partial}$-Bochner formula derived in \cite{Ni-1807} on $\|\partial f\|^2$ and $\|\Lambda^m \partial f\|_0^2$ to $\sigma_\ell$ and $\|\Lambda^\ell \partial f\|_0^2$. Since both $\sigma_\ell(x)$ and $\|\Lambda^\ell \partial f\|_0^2(x)$ are only continous in general we first derive formula on their barriers supplied by Proposition \ref{prop:21}, \ref{prop:22}.

\begin{proposition}\label{prop:31} Under the normal coordinates near $x_0$ and $f(x_0)$ such that $\partial f(\frac{\partial\ }{\partial z^\alpha})=\lambda_\alpha \delta^i_{\alpha} \frac{\partial\ }{\partial w^i}$ with $|\lambda_1|\ge \cdots\ge |\lambda_\alpha|\ge \cdots \ge |\lambda_m|$  being the singular values of $\partial f: (T'_{x_0}M, g) \to (T_{f(x_0)}'N, h)$, let $U_\ell(x)$ and $W_\ell(x)$ be  the functions defined in the last section in a small neighborhood of $x_0$. Then at $x_0$, for $v\in T'_{x_0}M$, and nonzero $U_\ell$ and $W_\ell$,
\begin{eqnarray}
\langle \sqrt{-1}\partial \bar{\partial} \log U_\ell, \frac{1}{\sqrt{-1}}v\wedge \bar{v}\rangle &=&\frac{U_\ell \sum_{1\le i\le n, 1\le \alpha \le \ell} |f^i_{\alpha v}|^2-|\sum_{\alpha =1}^\ell \overline{\lambda_\alpha}f^\alpha_{\alpha v}|^2}{U_\ell^2}\label{eq:31}\\
&\quad&+\sum_{\alpha=1}^\ell \frac{|\lambda_\alpha|^2}{U_\ell}(-R^N(\alpha, \bar{\alpha}, \partial f(v), \overline{\partial f(v)})+R^M(\alpha, \bar{\alpha}, v, \bar{v}));\nonumber \\
\langle \sqrt{-1}\partial \bar{\partial} \log W_\ell, \frac{1}{\sqrt{-1}}v\wedge \bar{v}\rangle &=& \sum_{\alpha=1}^\ell \sum_{ \ell+1\le i  \le  n} \frac{|f^i_{\alpha v}|^2}{|\lambda_\alpha|^2}\label{eq:32}  \\
&\quad&  + \sum_{\alpha=1}^\ell (- R^N(\alpha, \bar{\alpha}, \partial f(v), \overline{\partial f(v)})+R^M(\alpha,\bar{\alpha}, v, \bar{v})). \nonumber
\end{eqnarray}
\end{proposition}
\begin{proof} The calculation is similar to that of \cite{Ni-1807}. Here we include the details of the first. Choose holomorphic normal coordinate $(z_1, z_2, \cdots, z_m)$ near a point $p$ on the domain manifold $M$,  correspondingly  $(w_1, w_2, \cdots, w_n)$ near $f(p)$ in the target. Let $\omega_g=\sqrt{-1}g_{a\bar{\beta}}dz^\alpha\wedge d\bar{z}^{\beta}$ and $\omega_h=\sqrt{-1}h_{i\bar{j}}dw^i\wedge d\bar{w}^{j}$ be the K\"ahler forms  of $M$ and $N$ respectively. Correspondingly, the Christoffel symbols are given
$$
^M\Gamma_{\alpha \gamma}^\beta =\frac{\partial g_{\alpha \bar{\delta}}}{\partial z^{\gamma}}g^{\bar{\delta}\beta}=\Gamma_{\gamma \alpha }^\beta; \quad \quad  ^N\Gamma_{i k}^j =
\frac{\partial h_{i \bar{l}}}{\partial w^{k}}h^{\bar{l}k}=\Gamma_{k i }^j.
$$
We always uses Einstein's convention when there is an repeated index. The symmetry in the Christoffel symbols is due to K\"ahlerity. If the appearance of the indices can distinguish the manifolds we omit the superscripts $^M$ and $^N$.  Correspondingly the curvatures are given by
$$
^MR^\beta_{\alpha \bar{\delta} \gamma}=-\frac{\partial}{\partial \bar{z}^{\delta}} \Gamma_{\alpha \gamma}^\beta; \quad \quad \quad  \,^NR^j_{i \bar{l} k}=-\frac{\partial}{\partial \bar{w}^{l}} \Gamma_{i k}^j.
$$
At the points $x_0$ and $f(x_0)$, where the normal coordinates are centered we have that
$$
R_{\bar{\beta}\alpha \bar{\delta} \gamma}=-\frac{\partial^2 g_{\bar{\beta}\alpha}}{\partial z^{\gamma}\partial \bar{z}^{\delta}}; \quad \quad R_{\bar{j}i \bar{l} k}=-\frac{\partial^2 h_{\bar{j}i}}{\partial w^{k}\partial \bar{w}^{l}}.
$$
Direct calculation shows that at the point $x_0$ (here repeated indices $\alpha, \beta $ are summed from $1$ to $\ell$, while $i, j, k, l$ are summed from $1$ to $n$)
\begin{eqnarray*}
(\log U_\ell)_\gamma &=&\frac{g^{\alpha \bar{\beta}}_{\quad, \gamma} A_{\alpha\bar{\beta}}+g^{\alpha \bar{\beta}}f^i_{\alpha \gamma}h_{i\bar{j}}\overline{f^j_\beta}+g^{\alpha \bar{\beta}}f^i_{\alpha }\overline{f^j_\beta}f^k_\gamma h_{i\bar{j}, k} }{U_\ell}=\frac{f^i_{\alpha\gamma}\overline{f^i_\alpha}}{U_\ell};\\
(\log U_\ell)_{\bar{\gamma}} &=&\frac{g^{\alpha \bar{\beta}}_{\quad, \bar{\gamma}} A_{\alpha\bar{\beta}}+g^{\alpha \bar{\beta}}f^i_{\alpha}h_{i\bar{j}}\overline{f^j_{\beta\gamma}}+g^{\alpha \bar{\beta}}f^i_{\alpha }\overline{f^j_\beta}\overline{f^k_\gamma} h_{i\bar{j}, \bar{k}} }{U_\ell}=\frac{\overline{f^i_{\alpha\gamma}}f^i_\alpha}{U_\ell};\\
\left(\log U_\ell\right)_{\gamma \bar{\gamma}}&=& \frac{R^M_{\alpha\bar{\beta}\gamma\bar{\gamma}}f^i_\alpha \overline{f^i_\beta}+|f^i_{\alpha \gamma}|^2-R^N_{i\bar{j}k\bar{l}}f^i_\alpha \overline{f^j_\beta}f^k_\gamma \overline{f^l_\gamma}}{U_\ell}-\frac{|\sum_{1\le \alpha\le \ell; 1\le i\le n} f^i_{\alpha \gamma}\overline{f^i_\alpha}|^2}{U_\ell^2}.
\end{eqnarray*}
The claimed equation then follows.
\end{proof}

\begin{corollary} Let $f: M\to N$ be a holomorphic map between two K\"ahler manifolds.

(i) If the bisectional curvature of $N$ is non-positive and the bisectional curvature of $M$ is nonnegative, then $\log \sigma_\ell(x)$ is a plurisubharmonic function.

(ii) Assume that $\Ric^N_\ell\le 0$ and $\Ric^M_\ell\ge0$. If $\|\Lambda^\ell \partial f\|^2_0$ not identically zero, then for every $x$, there exists a $\Sigma \subset T_x'M$ with $\dim(\Sigma)\ge \ell$ such that $\log \|\Lambda^\ell \partial f\|^2_0(x)$ is plurisubharmonic on $\Sigma$.
\end{corollary}

\section{Proof of Theorem \ref{thm:main1}}
Since in general $\sigma_\ell$ and $\|\Lambda^\ell \partial f\|_0$ are not smooth we adopt the viscosity consideration as in Section 5 of \cite{Ni-1807} to prove the result. We also need to modify the algebraic argument in the Appendix of \cite{Ni-1807} for some point-wise estimates needed. Another difference of the argument is that we shall apply the maximum principle to a degenerate operator. First we need a Royden type lemma.

\begin{lemma}\label{lem:41} If the holomorphic sectional curvature $R^N$ has a upper bound $-\kappa$,  with respect
to the normal coordinates as in Proposition \ref{prop:21} at $x_0$ (and $f(x_0)$),
$$
\sum_{1\le \alpha,\beta, \gamma, \delta\le \ell} g^{\alpha \bar{\beta}}g^{\gamma \bar{\delta}}R^N_{i\bar{j}k\bar{l}}f^i_{\alpha} \overline{f^{j}_\beta} f^k_\gamma \overline{f^l_\delta} \le -\frac{\ell'+1}{2\ell'} \kappa U^2_\ell, \mbox{ when }\kappa>0; \quad  \le -\kappa U^2_\ell \mbox{ when } \kappa\le 0.
$$
Here $\ell'=\min\{ \ell, \dim(f(M))\}$.
\end{lemma}
\begin{proof} We follow the argument in Appendix of \cite{Ni-1807}, which is due to F. Zheng. The left hand side can be written as $ \sum_{1\le \alpha, \beta\le \ell'} R^N_{\alpha\bar{\alpha}\beta\bar{\beta}}|\lambda_\alpha|^2|\lambda_\beta|^2$. In the space $$\Sigma\doteqdot \operatorname{span} \{ \partial f\left( \frac{\partial\ }{\partial z^1}\right), \cdots, \partial f\left( \frac{\partial\ }{\partial z^{\ell'}}\right)\}, $$  consider the vector $Y=\sum_{1\le i \le \ell'} w^i\lambda_i \frac{\partial \ }{\partial w^i}$ with $(w^1, \cdots, w^{\ell'})\in \mathbb{S}^{2\ell'-1}\subset \Sigma$. Then  direct calculations show that
\begin{eqnarray*}
\sum_{1\le \alpha, \beta\le \ell'} R^N_{\alpha\bar{\alpha}\beta\bar{\beta}}|\lambda_\alpha|^2|\lambda_\beta|^2&=&\frac{\ell' (\ell'+1)}{2}\cdot \frac{1}{Vol (\mathbb{S}^{2\ell'-1})}\int_{\mathbb{S}^{2\ell'-1}} R^N(Y, \overline{Y}, Y,\overline{Y})\\
&\le & -\kappa \frac{\ell' (\ell'+1)}{2}\cdot \frac{1}{Vol (\mathbb{S}^{2\ell'-1})}\int_{\mathbb{S}^{2\ell'-1}} |Y|^4\\
&=& \frac{-\kappa }{2}\left(U_\ell^2+\sum_{1\le \alpha \le \ell'} |\lambda_\alpha|^4\right).
\end{eqnarray*}
The result follows from elementary inequalities $\sum_{1\le \alpha \le \ell'} |\lambda_\alpha|^4\le U_\ell^2 \le\ell'\, \sum_{1\le \alpha \le \ell'} |\lambda_\alpha|^4$.
\end{proof}

To prove part (i), let $\eta(t):[0, +\infty)\to [0, 1]$ be a function supported in $[0, 1]$ with $\eta'=0$ on $[0, \frac{1}{2}]$, $\eta' \le 0$, $\frac{|\eta'|^2}{\eta}+(-\eta'')\le C_1$. The construction of such $\eta$ is elementary.
Let $\varphi_R(x)=\eta(\frac{r(x)}{R})$.  When the meaning is clear we omit subscript $R$ in $\varphi_R$. Clearly $\sigma_\ell \cdot \varphi$ attains a maximum somewhere at $x_0$ in $B_p(R)$. With respect to the normal coordinates near $x_0$ and $f(x_0)$, $(U_\ell \varphi)(x_0)=(\sigma \varphi)(x_0)$, and $(U_\ell\varphi)(x)\le (\sigma_\ell  \varphi)(x)\le (\sigma\varphi)(x_0)\le (U_\ell \varphi)(x_0)$ for $x$ in the small normal neighborhood. The maximum principle then implies that at $x_0$
$$
\nabla (U_\ell \varphi)=0; \sum_{1\le \alpha \le \ell} \frac{1}{2}(\nabla_{\alpha}\nabla_{\bar{\alpha}} +\nabla_{\bar{\alpha}}\nabla_{\alpha}) \log (U_\ell \varphi) \le 0.
$$
Now applying the $\partial \bar{\partial}$ formula (\ref{eq:31}), the above Lemma and the complex Hessian comparison theorem of Li-Wang \cite{LW}, together with the argument in \cite{Ni-1807}, imply the result. It is clear from the proof that if $\ell=m=\dim(M)$, only the Laplacian comparison theorem is needed. Hence one only needs to assume  that the Ricci curvature of $M$ is bounded from below.

The proof of part (ii) is similar. For the sake of the completeness we include the argument under the assumption (\ref{eq:2}). In this case we let $\varphi=\eta\left(\frac{\rho}{R}\right)$. Now $\varphi$ has support in $D(2R)\doteqdot \{\rho\le 2R\}$. Hence $W_\ell \cdot \varphi$ attains its maximum somewhere, say at $x_0 \in D(2R)$. Now at $x_0$ we have \begin{eqnarray*}
0&\ge& \sum_{\gamma=1}^\ell \frac{\partial^2}{\partial z^\gamma \partial z^{\bar{\gamma}}}\, \left(\log (W_\ell \, \varphi)\right) \ge \sum_{\alpha,\gamma=1}^\ell R^M_{\alpha\bar{\alpha}\gamma \bar{\gamma}}-R^N_{\alpha\bar{\alpha }\gamma\bar{\gamma}}|\lambda_\gamma|^2 + \sum_{\gamma=1}^\ell \frac{\partial^2 \log \varphi}{\partial z^\gamma \partial z^{\bar{\gamma}}} \\
&\ge& -K +\ell \cdot \kappa \cdot W_\ell^{1/\ell}+\frac{ \eta''}{R^2\varphi} |\nabla \rho|^2+\frac{\ell \eta'}{R \varphi}\left([\sqrt{-1}\partial \bar{\partial} \rho]_{+}\right)-\frac{|\eta'|^2}{\varphi^2 R^2}\cdot |\nabla \rho|^2\\
&\ge&  -K +\ell \cdot \kappa \cdot W_\ell^{1/\ell} -\frac{C_1}{\varphi R^2}|\nabla \rho|^2 -\frac{C_1}{\varphi R} \cdot C(m)\left( [\sqrt{-1}\partial \bar{\partial} \rho]_{+}\right).
\end{eqnarray*}
Multiplying $\varphi$ on both sides of the above we have that
$$
\sup_{D(R)}\|\Lambda^\ell \partial f\|_0^2(x)\le \left(\frac{K+ +\frac{C_1}{\varphi R^2}|\nabla \rho|^2 +\frac{C_1}{\varphi R} \cdot C(m)\left( [\sqrt{-1}\partial \bar{\partial} \rho]_{+}\right)}{\ell \kappa}\right)^\ell.
$$
The result follows by observing that $\frac{|\nabla \rho|^2}{R^2}\le \frac{4|\nabla \rho|^2}{\rho^2} \to 0$ and $\frac{ [\sqrt{-1}\partial \bar{\partial} \rho]_{+}}{R}\le 2 \frac{ [\sqrt{-1}\partial \bar{\partial} \rho]_{+}}{\rho}\to 0$ as $R\to \infty$.

\section{Applications}
First we show that the Pogorelov type estimate  of  \cite{Ni-1807}  can be adapted to derive the $C^2$-estimate for the Monge-Amp\`ere equation related to the existence of K\"ahler-Einstein metrics and prescribing the Ricci curvature problem. Recall that the geometric problems reduce to a complex Monge-Amp\`ere equation
$$
\frac{\det(g_{\alpha\bar{\beta}}+\varphi_{\alpha\bar{\beta}})}{\det(g_{\alpha\bar{\beta}})}=e^{t\varphi +f}
$$
with $t\in [-1, 1]$, $f$ being a fixed function with prescribed complex Hessian. $g'_{\alpha\bar{\beta}}=g_{\alpha\bar{\beta}}+\varphi_{\alpha\bar{\beta}}$ is another K\"ahler metric with $[\omega_{g'}]=[\omega_g]$. We apply our previous setting to the map $\operatorname{id}:(M, g)\to (M, g')$. The computation in \cite{Aub, Yau} (See also the exposition in \cite{Siu}) is on $\mathcal{L} \|\partial f\|^2$. By the computation from Section 3 and 4, at the point where $\|\partial \operatorname{id}\|^2_0$ is attained we have that
$$
0\ge \frac{\partial^2}{\partial z^\gamma \partial \bar{z}^\gamma} \log (1+\varphi_{1\bar{1}})\ge R_{1\bar{1}\gamma\bar{\gamma}}-R'_{1\bar{1}\gamma \bar{\gamma}}\left(1+\varphi_{\gamma\bar{\gamma}}\right).
$$
Here $R'$ is the curvature of $g'$ and $|\lambda_\gamma|^2=1+\varphi_{\gamma\bar{\gamma}}$. Since we do not have information on $R'$ in general, but only $\Ric^{g'}(\frac{\partial}{\partial z^1}, \frac{\partial}{\partial \bar{z}^1})=\Ric^g_{1\bar{1}}-t\varphi_{1\bar{1}}-f_{1\bar{1}}$, we multiply $\frac{1}{1+\varphi_{\gamma\bar{\gamma}}}$ on the both sides of the above inequality and then sum $\gamma$ from $1$ to $m$ arriving at
\begin{eqnarray*}
0&\ge& \sum_{\gamma=1}^m\frac{1}{1+\varphi_{\gamma\bar{\gamma}}}R^g_{1\bar{1}\gamma\bar{\gamma}}-
\frac{\Ric^g_{1\bar{1}}}{1+\varphi_{1\bar{1}}}
+t\frac{\varphi_{1\bar{1}}}{1+\varphi_{1\bar{1}}}+\frac{f_{1\bar{1}}}{1+\varphi_{1\bar{1}}}\\
&\ge& -C(M, g, f)\sum_{\gamma=1}^m\frac{1}{1+\varphi_{\gamma\bar{\gamma}}}-1.
\end{eqnarray*}
 Now we apply/repeat the same consideration/calculation to $Q\doteqdot \log \sigma_1 -(C(M, g, f)+1)\varphi$.  Then  at the point $x_0$, where $Q$ attains its maximum,  we have that
$$
0\ge -C(M, g, f)\sum_{\gamma=1}^m\frac{1}{1+\varphi_{\gamma\bar{\gamma}}}-(C(M, g, f)+2) +(C(M, g, f)+1)\sum_{\gamma=1}^m\frac{1}{1+\varphi_{\gamma\bar{\gamma}}},
$$
which then implies that
$$
\sum_{\gamma=1}^m\frac{1}{1+\varphi_{\gamma\bar{\gamma}}}\le C(M, g, f)+2.
$$
This implies that at the maximum point of $\sigma_1 e^{-(C(M, g, f)+1)\varphi}$,
\begin{eqnarray*}
\sigma_1 e^{-(C(M, g, f)+1)\varphi}&=&\sigma_1\frac{\omega^m_g}{\omega^m_{g'}}e^{t\varphi +f}e^{-(C(M, g, f)+1)\varphi}\\
&\le& \left(\frac{1}{m-1} \sum_{\gamma=2}^m \frac{1}{1+\varphi_{\gamma\bar{\gamma}}}\right)^{m-1} e^{t\varphi +f}e^{-(C(M, g, f)+1)\varphi}\\
&\le& \left(\frac{C(M, g, f)+2}{m-1} \right)^{m-1} e^{t\varphi +f}e^{-(C(M, g, f)+1)\varphi}.
\end{eqnarray*}
If we write $K=\left(\frac{C(M, g, f)+2}{m-1} \right)^{m-1}$, $\kappa=C(M, g, f)+2$, the above implies
\begin{equation}
1+\varphi_{\gamma\bar{\gamma}}(x)\le \sigma_1(x)\le Ke^{\kappa (\varphi(x)-\varphi(x_0))} e^{t\varphi(x_0) +f(x_0)},  \quad \forall \gamma\in \{1, \cdots, m\}.
\end{equation}

As mentioned in the introduction, Theorem \ref{thm:main1} removes the constrains that $\dim(M)\le \dim(N)$ in the previous results proved in \cite{Ni-1807}. As in \cite{NZ} we  denote by $B^\perp$  the orthogonal bisectional curvature. We say $B^\perp\le \kappa$ if for any $X, Y\in T'N$ with $\langle X, \overline{Y}\rangle=0$, $R(X, \bar{X}, Y, \bar{Y})\le \kappa |X|^2|Y|^2$. The following is a corollary of the proof Theorem \ref{thm:main1}.

\begin{theorem}\label{thm:51} Let $f: (M, g)\to (N, h)$ be a holomorphic map.

(i) Assume that $M$ is compact. Under the assumptions either  $\Ric^M_\ell >0$, and the holomorphic sectional curvature $H^N\le 0$, or $\Ric^M_\ell\ge 0$ and $H^N < 0$,   $f$ must be constant.  The same result also holds if $(B^M)^\perp>0$ and $(B^N)^\perp\le 0$ or $(B^M)^\perp\ge0$ and  $(B^N)^\perp< 0$.

(ii) If $M$ is compact with $S^M_\ell\ge 0$ and $\Ric^N_\ell <0$, or  $S^M_\ell >0$, $\Ric^N_\ell \le 0$ then $\dim(f(M))< \ell$. The same result holds if $\Ric^M_\ell\ge 0$ and $S^N_\ell <0$, or $\Ric^M_\ell> 0$ and $S^N_\ell \le 0$.
\end{theorem}
\begin{proof} Since $M$ is compact $\sigma_\ell$ attains a maximum somewhere, say at $x_0$. If $f$ is not constant, $\sigma_\ell(x_0)>0$. Applying (\ref{eq:31}), using the normal coordinates around $x_0$ and $f(x_0)$ specified as in the last two sections we have that
$$
0\ge \sum_{\gamma=1}^\ell \frac{\partial^2\ }{\partial z^\gamma, \partial \bar{z}^\gamma}\left(\log U_\ell\right)\ge \sum_{1\le \alpha, \gamma \le \ell} \frac{-R^N_{\alpha\bar{\alpha}\gamma\bar{\gamma}}|\lambda_\alpha|^2|\lambda_\gamma|^2}{U_\ell}+ \sum_{\alpha =1}^\ell \frac{\Ric^M(x_0, \Sigma)(\alpha, \bar{\alpha})|\lambda_\alpha|^2}{U_\ell}.
$$
Here $\Sigma=\operatorname{span}\{ \frac{\partial\ }{\partial z^1}, \cdots, \frac{\partial\ }{\partial z^\ell} \}$. By Lemma \ref{lem:41}, if $H^N<0$, the first term is positive, the second one is nonnegative since $\Ric^M_\ell\ge 0$. Hence a contradiction. From the proof, the same holds if $H^N\le0$ and $\Ric^M_\ell>0$. For the case concerning $B^\perp$ the proof is similar.

For (ii), if $\operatorname{rank}(f)\ge  \ell$, $\|\Lambda^\ell\partial f\|_0$ has a nonzero maximum somewhere, say at $x_0$. Then applying (\ref{eq:32}), using the normal coordinates around $x_0$ and $f(x_0)$ specified as in the last two sections we have that
$$
0\ge \sum_{\gamma=1}^\ell \frac{\partial^2\ }{\partial z^\gamma, \partial \bar{z}^\gamma}\left(\log W_\ell\right)\ge \sum_{1\le  \gamma \le \ell}  (-\Ric^N_\ell(x_0, \Sigma) |\lambda_\gamma|^2)+\operatorname{Scal}^M(x_0, \Sigma).
$$
This leads to a contradiction under the assumptions either $S^M_\ell\ge 0$ and $\Ric^N_\ell <0$, or $S^M_\ell >0$, $\Ric^N_\ell \le 0$. For the second part, we introduce the operator:
$$
\mathcal{L}_\ell=\sum_{\gamma=1}^\ell\frac{1}{2|\lambda_{\gamma}|^2}\left(\nabla_{\gamma}\nabla_{\bar{\gamma}}
+\nabla_{\bar{\gamma}}\nabla_{\gamma}\right).
$$
Since at $x_0$ $\W_\ell\ne 0$, the above operator is well defined in a small neighborhood of $x_0$. As before applying $\mathcal{L}$ at $x_0$ implies that
$$
0\ge \mathcal{L}_\ell \left(\log W_\ell\right)\ge  (-\operatorname{Scal}^N(x_0, \partial f(\Sigma))+\sum_{1\le \gamma \le \ell} \frac{\Ric^M(x_0, \Sigma)(\gamma, \bar{\gamma})}{|\lambda_\gamma|^2}.
$$
The above also  induces a contradiction under either $\Ric^M_\ell\ge 0$ and $S^N_\ell <0$, or $\Ric^M_\ell> 0$ and $S^N_\ell \le 0$.
\end{proof}

This can be combined with the following result of Siu-Beauville (cf. Theorem 1.5 of \cite{ABCKT}) to infer information regarding the fundamental group of the manifolds with $\Ric_\ell\ge0$.

\begin{theorem}[Siu-Beauville] Let $M$ be a compact K\"ahler manifold. There exists a compact Riemann surface $C_g$ of genus greater than one and a surjective holomorphic map $f: M \to C'$ with $g(C')\ge g(C)$ with connected fibers if and only if there exists a surjective homomorphism $h: \pi_1(M)\to \pi_1(C_g)$.
\end{theorem}

\begin{corollary} (i) Let $(M, g)$ be a compact K\"ahler manifold with $\Ric_\ell\ge 0$ for some $1\le \ell\le m$. Then
there exists no surjective homomorphism $h: \pi_1(M)\to \pi_1(C_g)$. Furthermore, there is no subspace $V\subset H^1(M, \mathbb{C})$ with $\wedge^2 V=0$ in $H^2(M, \mathbb{C})$ and $\dim(V)\ge 2$. Namely $g(M)\le 1$. Similarly, if $\Ric_\ell\ge 0$,  $\pi_1(M)$ can not be of the type of an amalgamated product $\Gamma_1*_{\Delta}\Gamma_2$ with the index of $\Delta$ in $\Gamma_1$  greater than one and index of $\Delta$ in $\Gamma_2$ greater than two.

(ii) Let $(M, g)$ be a compact K\"ahler manifold with $S^M_\ell> 0$ for some $1\le \ell\le m$. Then $a(M)\le\ell-1$.

(iii) If $S^M_n\ge 0$, then any harmonic map $f: M\to N$  with $N$ being a locally Hermitian symmetric space,  can not have $\operatorname{rank}(f)=\dim(N)$.

\end{corollary}
\begin{proof} The first part of (i) follows from part (i) of Theorem \ref{thm:51}. Namely  apply it to $N=C_g$ and combine it with the above Siu-Beauville's result. The second part follows by combining Theorem \ref{thm:51} with Theorem 1.4 of \cite{ABCKT} due to Catanese  (cf. Theorem 1.10 of \cite{Cat}).  For the second part involving the amalgamated product, apply  Theorem 6.27 of \cite{ABCKT}, namely a result of Gromov-Schoen below instead,  to conclude that there exists an equivariant  holomorphic map from $\widetilde{M} $ into the Poincar\'e disk. This induce a contradiction with part (i) of Theorem \ref{thm:51} since the maximum principle argument still applies (see also \cite{NR}). The  statement of (ii) is an easy consequence of part (ii) of Theorem \ref{thm:51}.

For part (iii), by Siu's result on the holomorphicity of the harmonic maps between K\"ahler manifolds, namely Theorem 6.13 of \cite{ABCKT}, any such a harmonic map must be holomorphic. Then part (ii) of Theorem \ref{thm:51} induces a contradiction noting that the canonical metric on $N$ is K\"ahler-Einstein with negative Einstein constant.
\end{proof}

\begin{theorem}[Gromov-Schoen]
Let $M$ be a compact K\"ahler manifold with fundamental group $\Gamma=\Gamma_1*_{\Delta}\Gamma_2$ with the index of $\Delta$ in $\Gamma_1$  greater than one and index of $\Delta$ in $\Gamma_2$ greater than two. Then there exists a representation $\rho: \pi_1(M)\to \operatorname{Aut}(\mathbb{D})$, where $\mathbb{D}=\{z\, |\, |z|=1\}$, with discrete cocompact image, and a holomorphic equivariant map from the universal cover $\widetilde{M}\to \mathbb{D}$, which also descends to a surjective map $M\to \rho(\Gamma)/\mathbb{D}$.
\end{theorem}

In fact the vanishing theorem of \cite{Ni-Zheng2} implies that for K\"ahler manifolds with $S_\ell>0$, there does not exist a $k$-wedge subspace in $H^{1, 0}$ (in the sense of \cite{Cat}) for any $k\ge \ell$. Moreover, such manifolds have to be Albanese primitive for $k\ge \ell$.

For noncompact manifolds, Theorem \ref{eq:sch-ni} and Theorem \ref{thm:main1} can also be applied, together with Theorem 4.14 and 4.28 of \cite{ABCKT}, to infer some restriction on K\"ahler manifolds with nonnegative holomorphic sectional curvature or  with $\Ric_\ell\ge 0$.

\begin{corollary} Assume that $M$ is a complete K\"ahler manifold with bounded geometry with $\Ric^M_\ell\ge 0$. Then
(i) $H^1(M, \mathbb{C})=\{0\}$ implies that $\mathcal{H}^1_{L^2}(M)=\{0\}$;

(ii) And $\dim (\mathcal{H}^1_{L^2, ex}(M))\le 1$.
\end{corollary}

Here $\mathcal{H}_{L^2}(M)$ is the space of the harmonic $L^2$-forms and $\mathcal{H}^1_{L^2, ex}(M)$ is the space of the $L^2$ harmonic exact forms. The statements are trivial when $M$ is compact.

\section{Mappings from  positively curved manifolds}

  In \cite{NZ}, the orthogonal $\Ric^\perp$ was studied. Recall that $\Ric^\perp (X, \overline{X})=\Ric(X, \overline{X})-H(X)/|X|^2$. We say $\Ric^\perp\ge K$ if $\Ric^\perp (X, \overline{X})\ge K|X|^2$. It is easy to see that  $B^\perp\ge \kappa$ implies that $\Ric^\perp \ge (m-1)\kappa$. Similar upper estimate also holds if $B^\perp$ is bounded from above. It was also shown in \cite{NZ} via explicit examples that $B^\perp$ is independent of the holomorphic sectional curvature $H$, as well as the Ricci curvature. Similarly $\Ric^\perp$ is independent of $\Ric$, as well as $H$. It was proved in \cite{NZ} that for manifold whose  $\Ric^\perp$  has a  positive lower bound, the manifold is compact with an effective diameter uppper bound. (See \cite{Tsu} for the corresponding result for holomorphic sectional curvature.) It is not hard to see that for K\"ahler manifolds with $\Ric_\ell\ge K>0$, they must be compact with an upper diameter estimate.

Applying $\partial\bar{\partial}$-Bochner formulae we  have the following estimates in the spirit of \cite{Ni-1807}.

\begin{theorem}\label{thm:hoop}
 (i) Assume that  $\Ric^M_\ell (X, \overline{X})\ge K|X|^2$,   and     $H^N(Y)\le \kappa |Y|^4$, with $K, \kappa>0$. Then for any nonconstant $f: M\to N$
 $$
 \max_{x\in M} \sigma_\ell(x)\ge \frac{K}{\kappa}.
 $$
 (ii) Assume that  $(B^M)^{\perp}\ge K$,   and     $(B^N)^\perp\le \kappa$, with $K, \kappa>0$. Then for any nonconstant $f: M\to N$, $\dim(f(M))=m$. Moreover for any $\ell<\dim(M)$
 $$
 \max_{x\in M} \sigma_\ell(x)\ge \ell  \frac{K}{\kappa}.
 $$
 (iii) Assume that  $\Ric^M_\ell\ge K$, and that $\Ric^N_\ell \le \kappa $, with $K,\kappa>0$. Then for any  holomorphic map $f:M\to N$ with $\dim(f(M))\ge \ell$
 $$
 \max_{x}\|\Lambda^\ell \partial f\|_0^2(x) \ge \left(\frac{K}{\kappa}\right)^\ell.
 $$
 (iv) Assume that  $(\Ric^M)^\perp \ge K$, and that $(B^N)^\perp \le \kappa $, with $K,\kappa>0$. Then for any  holomorphic map $f:M\to N$ with $\dim(f(M))\ge m-1$, $\dim(f(M))=m$. Moreover
 $$
 \max_{x}\|\Lambda^m \partial f\|_0^2(x) \ge \left(\frac{K}{(m-1)\kappa}\right)^{m}.
 $$
 In the case $\dim(M)=\dim(N)$, only $(\Ric^N)^{\perp}\le (m-1)\kappa$ is needed. In general $(B^N)^\perp \le \kappa $ can be weakened to $(\Ric^N_m)^{\perp}\le (m-1)\kappa$. Here  $(\Ric^N_\ell)^{\perp}$ is the orthogonal Ricci curvature of the curvature tensor $R^N$ restricted to  $m$-dimensional subspaces.
\end{theorem}
\begin{proof} First observe that under any assumption of the  above theorem $M$ is compact. From Lemma \ref{lem:41} and (\ref{eq:31}), part (i) follows. For part (ii), at the point $x_0$ where $\sigma_\ell(x)$ attains its maximum, applying (\ref{eq:31}) to $v=\frac{\partial\ }{\partial z^m}$,  we have that
$$
0\ge -\kappa |\lambda_m|^2+K
$$
which implies that $|\lambda_m|^2\ge \frac{K}{\kappa}$. Then claimed estimate follows from $\sigma_\ell\ge \ell |\lambda_m|^2$.

 For part (iii), we apply (\ref{eq:32}) at the point $x_0$, where $\|\Lambda^\ell \partial f\|_0^2(x)$ attains its  maximum. In particular we apply it to $v=\frac{\partial\ }{\partial z^\ell}$ and let $\Sigma=\operatorname{span}\{ \frac{\partial\ }{\partial z^1}, \cdots, \frac{\partial\ }{\partial z^\ell}\}$. Hence at $x_0$
 $$
 0\ge -\Ric^N(x_0, f(\Sigma)) |\lambda_\ell|^2+\Ric^M (x_0, \Sigma).
 $$
 Hence we derive that $|\lambda_\ell|^2\ge \frac{K}{\kappa}$. The claimed result then follows.

 The part (iv) can be proved similarly.
\end{proof}

The part (ii) of the theorem is not as strong as it appears, since $B^\perp>0$ implies that $h^{1,1}(M)=1$. On the other hand we have the following observation.

\begin{proposition}
 Let $M$ be a K\"ahler manifold with $h^{1,1}(M)=1$. Then any holomorphic map $f: M\to N$, with $\dim(f(M))<\dim(M)$ must be a constant map. Hence $g(M)\le 1$, if $\dim(M)\ge 2$. In particular, if  the Picard number $\rho(M)=1$ and $S_2^M>0$, any holomorphic map $f: M\to N$, with $\dim(f(M))<\dim(M)$ must be a constant map.
\end{proposition}
\begin{proof}
In fact $f^*\omega_h$, with $\omega_h$ being the K\"ahler form of $N$, is a $d$-closed positive $(1,1)$-form. By the assumption $[f^*\omega_h]$ proportional to $[\omega_g]$.  Hence it must be either zero or a positive multiple of $[\omega_g]$. Since the second case implies that $\dim(f(M))=m$, only the first case can occur, which implies that $f$ is a constant map.

Note that this implies that for any K\"ahler manifold $M$ with $\dim(M)\ge 2$ and $h^{1,1}(M)=1$, the genus $g(M)\le 1$, in view of the result of Catanese (cf. Theorem 1.10 of \cite{Cat}) since otherwise there exists a nonconstant  holomorphic map $f: M\to C_g$ with $C_g$ being a Riemann surface of genus $g(M)$. Since the first Chern class map $c_1: H^{1}(M, \mathcal{O}^*)\to \mathcal{H}^{1,1}(M)\cap H^2(M, \mathbb{Z})$ is onto, and $S^M_2>0$ implies that $H^2(M, \mathbb{C})=\mathcal{H}^{1,1}(M)$, the assumption then implies $h^{1,1}(M)=1$. The last result then follows from the first.
\end{proof}
Taking $\kappa\to 0$, the part (ii) of Theorem \ref{thm:hoop} also implies that any holomorphic map from  a compact manifold with $B^\perp>0$ into one with $B^\perp\le 0$ must be a constant map (cf. Theorem \ref{thm:51}). Given that $B^\perp$ is independent of $H$ and $\Ric$, this does not follow from Yau-Royden's estimate Theorem \ref{thm-sch-roy}, nor from Theorem \ref{thm:sch1}.  The part (iv) provides an additional information on compact K\"ahler manifolds with $\Ric^\perp>0$.

\section*{Acknowledgments} {We would like to thank James McKernan (particularly bringing my attention to the work \cite{Laz}) and Fangyang Zheng for conversations regarding holomorphic maps from $\mathbb{P}^m$. We are also grateful to Yanyan Niu for informing \cite{Mu}.}


\begin{thebibliography}{A}


\bibitem{ABCKT} J. Amor\'os, M. Burger, K. Corlette, D. Kotschick, and D. Toledo, \textit{ Fundamental groups of compact K\"ahler manifolds.} Mathematical Surveys and Monographs, \textbf{44}. American Mathematical Society, Providence, RI, 1996.


\bibitem{Aub} T.  Aubin, \textit{ Nonlinear analysis on manifolds. Monge-Amp\`ere equations.} Grundlehren der Mathematischen Wissenschaften [Fundamental Principles of Mathematical Sciences], \textbf{252}. Springer-Verlag, New York, 1982.


\bibitem{Cat} F. Catanese, \textit{  Moduli and classification of irregular K\"ahler manifolds (and algebraic varieties) with Albanese general type fibrations.} Invent. Math. \textbf{104} (1991), no. 2, 263--289.

\bibitem{CG} J. Cheeger and D. Gromoll, \textit{ The splitting theorem for manifolds of nonnegative Ricci curvature.} J. Differential Geometry \textbf{6} (1971/72), 119--128.

\bibitem{Fede} H. Federer, \textit{ Geometric Measure Theory. } Die Grundlehren der mathematischen Wissenschaften, Band 153 Springer-Verlag New York Inc., New York 1969.


\bibitem{FW} A. Fraser and J. Wolfson, \textit{ The fundamental group of manifolds of positive isotropic curvature and surface groups.} Duke Math. J. \textbf{133} (2006), no. 2, 325--334.



\bibitem{Gu} C. E. Guti\'errez, \textit{ The Monge-Amp\`ere Equation.} 2nd edition. Progress in Nonlinear Differential Equation and Their Applications.  \textbf{89} (2016), Birkh\"auser.


 \bibitem{Hitchin}
N. Hitchin, \textit{ On the curvature of rational surfaces.} In Differential Geometry ({\em Proc. Sympos. Pure Math., Vol XXVII, Part 2, Stanford University, Stanford, Calif., 1973}), pages 65-80. Amer. Math. Soc., Providence, RI, 1975.


\bibitem{Horn-Johnson} R. Horn and C. Johnson, \textit{ Matrix Analysis.} 2nd Edition. Cambridge University Press, 2013.





\bibitem{Kobayashi-H}  S. Kobayashi,    \textit{Hyperbolic Complex Spaces.} Springer, New York, 1998.

\bibitem{Laz} R. Lazarsfeld, \textit{Some applications of the theory of positive vector bundles. Complete intersections} (Acireale, 1983), 29–61, Lecture Notes in Math., 1092, Springer, Berlin, 1984.


 \bibitem{LW} P. Li and J.-P.  Wang,  \textit{Comparison theorem for K\"ahler manifolds and positivity of spectrum.} J. Differential Geom. \textbf{69} (2005), no. 1, 43--74.




\bibitem{MM} M. Marcus and H. Minc, \textit{A survey of matrix theory and matrix inequalities.} Reprint of the 1969 edition. Dover Publications, Inc., New York, 1992.

\bibitem{Mok} N. Mok, \textit{ The uniformization theorem for compact K\"ahler manifolds of nonnegative holomorphic bisectional curvature.} J. Differential Geom. \textbf{27} (1988), no. 2, 179--214.


\bibitem{Mu} S. Matsumura, \textit{ On projective manifolds with semi-positive holomorphic sectional curvature.} ArXiv:1811.04182.


\bibitem{NR} T. Napier and M.  Ramachandran, \textit{Filtered ends, proper holomorphic mappings of K\"ahler manifolds to Riemann surfaces, and K\"ahler groups.} Geom. Funct. Anal. \textbf{17} (2008), no. 5, 1621--1654.





\bibitem{Ni-1807}L. Ni, \textit{ Liouville  theorems and a Schwarz Lemma for holomorphic mappings between K\"ahler manifolds.} ArXiv preprint: 1807.02674.


 \bibitem{Ni} L. Ni, \textit{The fundamental group, rational connectedness and the positivity of K\"ahler manifolds.} ArXiv preprint:1902.00974

 \bibitem{Ni-Tam} L. Ni and L.-F.  Tam, \textit{ Plurisubharmonic functions and the structure of complete K\"ahler manifolds with nonnegative curvature.} J. Differential Geom. \textbf{64} (2003), no. 3, 457--524.

\bibitem{NZ} L. Ni and F. Zheng, \textit{ Comparison and vanishing  theorems for K\"ahler manifolds.} Calc. Var. Partial Differential Equations, \textbf{57}(2018), no. 6, Art. 151, 31 pp.

\bibitem{Ni-Zheng2} L. Ni and F. Zheng, \textit{ Positivity and Kodaira embedding theorem.} ArXiv preprint:1804.09696.




\bibitem{Pogo} A.-V.  Pogorelov, \textit{ The Minkowski multidimensional problem.} Translated from the Russian by Vladimir Oliker. Introduction by Louis Nirenberg. Scripta Series in Mathematics. V. H. Winston \& Sons, Washington, D.C.; Halsted Press [John Wiley \& Sons], New York-Toronto-London, 1978.


\bibitem{Roy} H. L. Royden, \textit{ The Ahlfors-Schwarz lemma in several complex variables.} Comment. Math. Helv. \textbf{55} (1980), no. 4, 547--558.



\bibitem{Siu} Y.-T. Siu, \textit{ Lecture on Hermitian-Einstein Metrics for Stable Bundles and K\"ahler-Einstein Metrics.} Birkh\"auser, Basel, 1987.





\bibitem{Tsu}  Y.   Tsukamoto,\textit{ On K\"ahlerian manifolds with positive holomorphic sectional curvature. } Proc. Japan Acad. \textbf{33} (1957), 333--335.
	

\bibitem{Whit} H.
Whitney, \textit{ Geometric Integration Theory.} Princeton University Press, Princeton, N. J., 1957.



\bibitem{Yau} S.-T. Yau, \textit{ On Ricci curvature of a compact K\"ahler manifold and the complex Monge-Amp\`ere equation, I,} Comm. Pure and Appl. Math. \textbf{31} (1978), 339--411.


\bibitem{Yau-sch} S. T. Yau, \textit{ A general Schwarz lemma for K\"ahler manifolds.} Amer. J. Math. \textbf{100}  (1978), no. 1, 197--203.


\end{thebibliography}
\end{document}